\documentclass[11pt]{amsart}
\usepackage{graphicx}              
\usepackage{amsmath}              
\usepackage{amsfonts}              
\usepackage{amsthm}                
\usepackage{color}
\usepackage{tabulary}
\usepackage{caption}
\usepackage{subcaption}
\usepackage{amssymb}
\usepackage{todonotes}
\theoremstyle{plain}
\usepackage{mathtools}

\theoremstyle{plain}
\newtheorem{thm}{Theorem}
\newtheorem*{cor}{Corollary}
\newtheorem{theorem}{Theorem}[section]

\newtheorem{proposition}[theorem]{Proposition}
\newtheorem{lemma}[theorem]{Lemma}

\newtheorem{corollary}[theorem]{Corollary}
\theoremstyle{definition}

\newtheorem{remark}[theorem]{Remark}

\newcommand{\nc}{\newcommand}
\nc{\dmo}{\DeclareMathOperator}

\nc{\Q}{\mathbb{Q}}
\nc{\R}{\mathbb{R}}
\nc{\Z}{\mathbb{Z}}
\nc{\N}{\mathbb{N}}
\nc{\C}{\mathbb{C}}
\nc{\cS}{\mathcal{S}}
\nc{\iso}{\cong}
\dmo{\Mod}{Mod}
\dmo{\Diff}{Diff}
\dmo{\Homeo}{Homeo}
\dmo{\dist}{dist}
\dmo\BDiff{BDiff}
\dmo\SO{SO}
\dmo\slide{sl}
\dmo\im{im}
\dmo\id{id}
\dmo\Fix{Fix}
\dmo\Out{Out}
\dmo{\T}{\mathcal{T}}
\dmo{\Te}{\mathcal{T}^{\epsilon}}
\dmo{\M}{\mathcal{M}}
\dmo{\Me}{\mathcal{M}^{\epsilon}}

\renewcommand{\epsilon}{\varepsilon}
\nc{\coloneq}{\mathrel{\mathop:}\mkern-1.2mu=}
\nc{\margin}[1]{\marginpar{\scriptsize #1}}
\nc{\para}[1]{\bigskip\noindent\textbf{#1}}

\newcommand{\cal}{\mathcal}

\begin{document}

\title[Small eigenvalues and thick-thin decomposition]
{Small eigenvalues and thick-thin decomposition in negative curvature}
\author{Ursula Hamenst\"{a}dt}

\thanks
{
Spectrum of the Laplacian, Neumann boundary conditions, 
manifolds of pinched negative curvature, 
thick-thin decomposition\\
Research
supported by ERC grant ``Moduli''}
\date{November 14, 2019}

\begin{abstract}
Soit $M$ une vari\'et\'e Riemannienne compl\`{e}te orient\'ee,
de dimension $n\geq 3$ et de volume finie. Supposons que la
courbure de $M$ soit contenue dans $[-b^2,-1]$, et soit
$M=M_{\rm thick}\cup M_{\rm thin}$ la d\'ecomposition 
en sa partie \'epaisse et sa partie fine. Soit $\lambda_k(M)$ la 
$k$-ti\`{e}me valeur
propre de l'op\'erateur Laplacien, avec conditions de bord de Neumann. 
Nous d\'emontrons que 
$\lambda_k(M_{\rm thick})/3\leq \lambda_k(M)$ pour tout $k$ tel que
$\lambda_k(M)<(n-2)^2/12$. 
Si $M$ est hyperbolique et de dimension 3, alors $\lambda_k(M)\leq 
C\log({\rm vol}(M_{\rm thin})+2)\lambda_k(M_{\rm thick})$ pour un nombre
$C>0$ fix\'e pourvu que
$\lambda_k(M_{\rm thick})<1/96$.

\bigskip

Let $M$ be a finite volume oriented complete Riemannian
manifold of dimension $n\geq 3$ and curvature in $[-b^2,-1]$, 
with thick-thin decomposition $M=M_{\rm thick}\cup M_{\rm thin}$.
Denote by $\lambda_k(M_{\rm thick})$ 
the $k$-th eigenvalue for the Laplacian on $M_{\rm thick}$, with 
Neumann boundary conditions. We show  
that $\lambda_k(M_{\rm thick})/3\leq \lambda_k(M)$ for all $k$ for which 
$\lambda_k(M)<(n-2)^2/12$. 
If $M$ is hyperbolic and of dimension 3 then $\lambda_k(M)\leq 
C\log({\rm vol}(M_{\rm thin})+2)\lambda_k(M_{\rm thick})$ for a fixed
number $C>0$ 
provided that $\lambda_k(M_{\rm thick})<1/96$.
\end{abstract}

\maketitle


\section{Introduction}

The small part of the spectrum of the Laplace operator $\Delta$ acting on 
functions on a closed oriented 
hyperbolic surface $S$ is quite well understood. Namely,
if $g$ denotes the genus of $S$ then the $2g-3$-th eigenvalue 
$\lambda_{2g-3}(S)$ can be arbitrarily
small \cite{Bu}, while $\lambda_{2g-2}(S)>\frac{1}{4}$ \cite{OR09}.

By the Gauss-Bonnet theorem, the volume of $S$ equals $2\pi(2g-2)$, so
these results relate the small part of the spectrum of $S$ to its volume.

For $n\geq 3$, the small part of the spectrum of an oriented  
finite volume Riemannian 
manifold $M$ of dimension $n$ and 
sectional curvature $\kappa\in [-b^2,-1|$ for some $b\geq 1$ 
is less well understood. 
The manifold $M$ admits a \emph{thick-thin decomposition}
$M=M_{\rm thick}\cup M_{\rm thin}$ 
which is determined as follows. There exists a constant
$c=c(n)>0$ such that $\epsilon=b^{-1}c(n)$ is a \emph{Margulis constant} 
for $M$, and 
$M_{\rm thick}$ is the set of all points in $M$ of injectivity radius 
at least $\epsilon$. 
Its complement $M_{\rm thin}$ is a disjoint union of
so-called \emph{Margulis tubes} and \emph{cusps}. 
Here a Margulis tube is a tubular neighborhood of a closed geodesic 
of length at most $2\epsilon$, and a cusp is homeomorphic to the quotient of 
a horoball in the universal covering $\tilde M$ of $M$ 
by a rank $n-1$ parabolic subgroup of the isometry group of $\tilde M$. 
After a small modification, 
we may assume that
$M_{\rm thick}=M-M_{\rm thin}$ is a submanifold of $M$ with smooth boundary
\cite{BCD93}.

Unlike in the case of surfaces, 
the submanifold $M_{\rm thick}$ is always connected, and
this prevents the occurence of very small eigenvalues. 
More precisely, for every $n\geq 3$, 
Schoen \cite{S82} established the existence of a universal and explicit 
constant $\theta=\theta(n,b)>0$ such that
\begin{equation}\label{lower}\lambda_1(M)\geq 
\frac{\theta}{{\rm vol}(M)^2}\end{equation}
for \emph{every} closed $n$-manifold $M$ with curvature in 
$[-b^2,-1]$ (here in contrast to the work of Schoen, we 
normalize the metric on $M$ so that the upper bound 
of the curvature is fixed). That 
this estimate extends without
change to non-compact finite volume manifolds
was established in \cite{DR86,D87}.

For hyperbolic manifolds of dimension $n=3$, this bound is roughly sharp:  
White \cite{White} 
proved that for fixed $r \geq 2, \delta > 0$ 
there exists a constant 
 $a=a(r,\delta)>0$ such that 
 \[ \lambda_1(M)\in [1/a {\rm vol}(M)^2, a/{\rm vol}(M)^2]\]
 for any closed 
hyperbolic 3-manifold whose injectivity radius is bounded from 
below by $\delta$ and such 
that the rank of the fundamental group $\pi_1(M)$ of $M$ is at most $r$.
This rank is defined to be the minimal number of generators of 
$\pi_1(M)$, and it 
is bounded from above by a fixed multiple of the
volume (Theorem 1.10 in \cite{G04}). A similar statement also holds
true for random 3-dimensional hyperbolic mapping tori \cite{BGH}
and for random hyperbolic 3-manifolds of fixed Heegaard genus
\cite{HV19}.

Since the injectivity radius at points in $M_{\rm thick}$ is 
at least $\epsilon$, the submanifold $M_{\rm thick}$ of $M$ is 
uniformly
quasi-isometric to a finite graph $G$ of uniformly 
bounded valence, with constant
only depending on the dimension and the curvature bounds 
(see \cite{BGLM02} for a detailed discussion of this fact 
in the case of hyperbolic manifolds). 
If $\vert G\vert $ denotes 
the number of vertices of $G$, 
then for $k\leq \vert G\vert $ the $k$-th eigenvalue
$\lambda_k(M_{\rm thick})$ of $M_{\rm thick}$ with Neumann boundary
conditions is uniformly comparable to the $k$-th eigenvalue 
$\lambda_k(G)$ 
of the graph 
Laplacian of $G$ \cite{Man05}. 
Here and in the sequel, we list eigenvalues as 
$0=\lambda_0<\lambda_1\leq \lambda_2\leq \dots$
with each eigenvalue repeated according to its multiplicity. 
Note that $\vert G\vert $ is proportional to the 
volume ${\rm vol}(M_{\rm thick})$ of $M_{\rm thick}$, 
with multiplicative constants only depending on $n$ and $b$.

Now for any graph $G$, there are precisely $\vert G\vert$ eigenvalues
$0=\lambda_0(G)<\lambda_1(G)\leq\dots\leq \lambda_{\vert G\vert -1}(G)$, and
(Lemma 1 of \cite{Chu96})  
\[\sum_i\lambda_i(G)=\vert G\vert.\]
In particular, we have $\lambda_{\vert G\vert-1}(G)\geq 1$.
As the volume of a finite volume oriented manifold $M$ of curvature in 
$[-b^2,-1]$ is uniformly comparable to the volume of $M_{\rm thick}$, this 
together with the result in \cite{Man05} implies 
that there exists a constant
$q=q(n,b)>1$ so that $\lambda_{q{\rm vol}(M)}(M_{\rm thick})\geq 1/q$. 

To recover a relation between the spectrum of $M$ and the
volume of $M$ which resembles the result known for hyperbolic surfaces,
it is therfore  
desirable
to relate the small part of the spectrum of $M$ to the small part
of the spectrum of $M_{\rm thick}$, taken with Neumann boundary condition.
The first purpose of this
article is to establish such a relation. We show

\begin{thm} \label{thm1}
For a suitable choice of a Margulis constant, we have 
\[\lambda_k(M)\geq \min\{\frac{1}{3}\lambda_k(M_{\rm thick}),\frac{(n-2)^2}{12}\}\]
for every finite volume oriented Riemannian manifold $M$ of dimension
$n\geq 3$ and curvature 
$\kappa\in [-b^2,-1]$, and all $k\geq 0$. 
\end{thm}

The dependence of this estimate 
on the lower curvature bound $-b^2$
enters this result via the Margulis constant which depends on $b$.
Note also that the constant $(n-2)^2/12$ is 
uniformly comparable to
the lower bound $(n-1)^2/4$ for  
the bottom of the essential spectrum of a non-compact
finite volume manifold of curvature $\kappa\in [-b^2,-1]$
(Corollary 3.2 of \cite{H04}). 
We expect that our methods can be used to extend
Theorem \ref{thm1} 
to geometrically finite manifolds of infinite volume and
curvature in $[-b^2,-1]$. This could then be used to establish
an improvement of Corollary 3.3 of \cite{H04} which contains
the following statement as a special case: 
The number of eigenvalues
contained in $(0,(n-2)^2/12)$ is at most $d^{{\rm vol}(M)}$ for a fixed 
constant $d>0$. However, we do not attempt to carry 
out such a generalization in this article. 

As an application of Theorem \ref{thm1}, we recover the results of Schoen, 
of Dodziuk and of Randol and Dodziuk, and we relate 
the number of small 
eigenvalues to the volume 
as promised. 

\begin{cor}\label{application}  
For all $n\geq 3,b\geq 1$ there exists a constant 
$\chi=\chi(n,b)>0$ with the
following property. 
Let $M$ be a finite volume oriented Riemannian 
manifold of dimension $n$ and curvature $\kappa\in [-b^2,-1]$; then 
\begin{enumerate}
\item
$\lambda_1(M)\geq \frac{\chi}{{\rm vol}(M)^2}.$
\item $\lambda_{ {\rm vol}(M)/\chi}(M)\geq \chi$.
\end{enumerate}
\end{cor}

Unlike in the work of Schoen, the constant $\chi(n,b)$ 
in the above Corollary 
is not explicit as it 
depends on a Margulis constant for Riemannian 
manifolds with curvature in $[-b^2,-1]$.
However, for hyperbolic manifolds
it can explicitly be estimated. 

For hyperbolic 3-manifolds $M$ we also obtain upper bounds  
for the small eigenvalues of $M$. We show

\begin{thm}\label{thm2}
There exists a number $c>0$ such that for every finite volume 
oriented hyperbolic 3-manifold $M$, we have  
\[\lambda_k(M)\leq c\log ({\rm vol}(M_{\rm thin})+3)\lambda_k(M_{\rm thick})\]
for all $k\geq 1$ such  that $\lambda_k(M_{\rm thick})<1/96$.
\end{thm}



The proof of Theorem \ref{thm1} uses standard comparison results and a
simple decomposition principle, and it is 
carried out in Section 2. Theorem \ref{thm2}
is shown in Section \ref{above}
with an explicit construction which is only 
valid for hyperbolic 3-manifolds.

\bigskip\noindent 
{\bf Acknowledgement:} I am grateful to Juan Souto for
inspiring discussions. Some versions of the results in this paper
were independently obtained by 
Anna Lenzhen and Juan Souto  \cite{LS16}.


\section{Bounding small eigenvalues from below}\label{below}

The goal of this section is to show Theorem \ref{thm1}.

Thus let $M$ be a finite volume oriented Riemannian 
manifold of dimension $n\geq 3$
and sectional curvature $\kappa\in [-b^2,-1]$ for some 
$b\geq 1$.  
Then $M$ admits a \emph{thick-thin decomposition}
\[M=M_{\rm thin}\cup M_{\rm thick}.\]
For a number $\epsilon=b^{-1}c(n)>0$, a so-called \emph{Margulis constant}, 
the thin part $M_{\rm thin}$ is the set of all 
points $x\in M$ with injectivity
radius ${\rm inj}(x)\leq  \epsilon$,
and $M_{\rm thick}=\{x\mid {\rm inj}(x)\geq \epsilon\}$.
The set 
$M_{\rm thick}$ is a non-empty 
compact connected manifold with (perhaps non-smooth) boundary, and  
$M_{\rm thin}$ is a union of (at most) finitely many 
\emph{Margulis tubes} and \emph{cusps}.

A Margulis tube is a tubular neighborhood of a closed geodesic
$\gamma$ in $M$ of length smaller than $2\epsilon$, and it is 
homeomorphic to $B^{n-1}\times S^1$
where $B^{n-1}$ is the closed unit ball in $\mathbb{R}^{n-1}$.
The geodesic $\gamma$ 
is called the \emph{core geodesic} of the tube. 

Let $T$ be such a Margulis tube, with core geodesic $\gamma$
of length $\ell<2\epsilon$. 
We fix a parameterization of $\gamma$ by arc length on the interval
$[0,\ell)$. 
Let $\sigma$ be the standard angular coordinate on the fibers of the 
unit normal bundle $N(\gamma)$ 
of $\gamma$ in $M$ obtained by parallel transport of 
the fibre over $\gamma(0)$  
(this unit normal bundle is an $S^{n-2}$-bundle
over $\gamma$), let $s$ be the length parameter
of $\gamma$ and let $\rho\geq 0$ be the radial distance from $\gamma$.
Via the normal exponential map, these functions define 
"coordinates" (i.e. a parameterization) $(\sigma,s,\rho)$ for $T-\{\gamma\}$,
defined on an open subset of $N(\gamma)\times (0,\infty)$
which will be specified below. 
In these "coordinates",
the maps $\rho\to (\sigma, s, \rho)$ are unit speed geodesics 
with starting point on $\gamma$ and initial velocity
perpendicular to $\gamma^\prime$.

There exists a continuous function 
\[R:N(\gamma)\to (0,\infty), (\sigma,s)\to R(\sigma,s)\]
such that in these "coordinates", we have
$T=\{\rho\leq R\}$. The metric on $T-\{\gamma\}$ is 
of the form
$h(\rho) +d\rho^2$ where $h(\rho)$ 
is a family of smooth metrics on the hypersurfaces $\rho={\rm const}$.

Lemma 2.4 of \cite{BCD93} states that there exists a constant
$\nu(n,b)>0$ which can be computed as an explicit function 
of the constants $c(n), b,n$ such that
\begin{equation}\label{mincontrol}
\min R\geq -\log \ell -\nu(n,b).\end{equation}
In particular, up to slightly adjusting the thick-thin decomposition
and replacing $M_{\rm thick}$  
by its union with all Margulis tubes with core geodesics of length
$\ell$ so that $\log \ell\geq -3-\nu(n,b)$, 
we may assume that 
for every component $T$ of $M_{\rm thin}$, the distance between
the core geodesic and the boundary $\partial T$ is at least three. 

In general, the boundary of a Margulis tube need not be smooth.
However, 
Theorem 2.14 of \cite{BCD93} shows that it can be perturbed 
to be smooth and of controlled geometry. 
We record this result for completeness.

\begin{theorem}\label{thmbcd93}
Let $T\subset M$ be a Margulis tube, with core geodesic $\gamma$ and 
boundary $\partial T$.
Then there exists a smooth hypersurface $H\subset T-\gamma$ with
the following properties.
\begin{enumerate}
\item The angle $\theta$ between the tangent of the radial geodesic
and the exterior normal to $H$ is less than $\pi/2-\alpha$
for some $\alpha=\alpha(n,b)\in (0,\pi/2)$.
\item The sectional curvatures of $H$ with respect to the induced
metric are bounded in absolute value by a constant depending only
on $n$ and $b$.
\item $H$ is homeomorphic to $\partial T$ by pushing along radial
arcs. The distance between $x\in H$ and its image 
$\bar x\in \partial T$ satisfies $d(x,\bar x)\leq bc(n)/50$.
\end{enumerate} 
\end{theorem}

In the sequel we always assume that the boundary $\partial T$ of 
a Margulis tube has the properties stated in 
Theorem \ref{thmbcd93}.
Then the injectivity radius of $\partial T$ 
with respect to the induced metric is bounded from below by
a positive constant only depending on the curvature bound and
the dimension (Corollary 2.24 of \cite{BCD93}).

Our first goal is to obtain a better understanding of the 
volume element of a Margulis tube. To this end 
we record a variant of Lemma 1 
of \cite{DR86}. 
We begin with a simple comparison lemma. 

Let $\tilde M$ be a simply connected complete Riemannian manifold 
of dimension $n$, with sectional curvature
$\kappa\leq -1$. Let $\eta:\mathbb{R}\to \tilde M$ be a geodesic parametrized by 
arc length and let $Y_1(t),\dots,Y_{n-1}(t)$ be orthonormal parallel
vector fields along $\eta$, orthogonal to $\eta^\prime$. Let moreover 
$J_1,\dots,J_{n-1}$ be Jacobi fields along $\eta$,
orthogonal to $\eta^\prime$, with the following initial conditions.

\begin{enumerate}
\item $J_1(0)=Y_1(0)$, and 
$\frac{\nabla}{dt}J_1(0)=0$.
\item For $i=2,\dots,n-1$ 
we have  
$J_i(0)=0$, and covariant derivatives 
$\frac{\nabla}{dt} J_i(t)\vert_{ t=0}=Y_i(0)$. 
\end{enumerate}

Then $J(t)=(J_1(t),\dots,J_{n-1}(t))$ can be viewed as an 
$(n-1,n-1)$-matrix with respect to the basis $Y_1(t),\dots,Y_{n-1}(t)$.  
Define the matrix $A(t)$ by 
\begin{equation}\label{riccati}
\frac{\nabla}{dt}J(t)=A(t)J(t).\end{equation}
The matrix valued map $t\to A(t)$ satisfies the \emph{Riccati equation}
\[A^\prime +A^2+R_{\eta^\prime}=0\]
where $R_{\eta^\prime}$ is the curvature tensor of $\tilde M$ 
evaluated on $\eta^\prime$. In particular,
$A$ is self-adjoint.

Using the notations from \cite{E94}, 
for $t>0$ we denote by $j(t)$ the determinant of the matrix $J(t)$.

\begin{lemma}\label{volumegrowth}
For all $R>0$ we have 
\[j(R)\geq (n-1)\int_0^Rj(t)dt.\]
\end{lemma}
\begin{proof}
The proof follows from standard comparison. 
We use the above notations. 

Let $\bar \eta:\mathbb{R}\to \mathbb{H}^n$ be a geodesic in the hyperbolic $n$-space $\mathbb{H}^n$.
Let $\bar Y_1(t),\dots,\bar Y_{n-1}(t)$ be parallel vector fields along $\bar\eta$
such that for each $t$, $\bar Y_1(t),\dots,\bar Y_{n-1}(t)$ is an orthonormal basis of 
$\bar \eta^\prime(t)^\perp$. 
Let $\bar J_i$ $(1\leq i\leq n-1)$ be the Jacobi fields along $\bar\eta$ defined by 
$\bar J_1(0)=\bar Y_1(0),\frac{\nabla}{dt}\bar J_1(0)=0$, 
and for $i\geq 2$ we require that
$\bar J_i(0)=0$ and $\frac{\nabla}{dt}\bar J_i(0)=\bar Y_i(0)$. Write 
$\bar J(t)=(\bar J_1(t),\dots,\bar J_{n-1}(t))$ and view this as a matrix with respect to the
basis $\bar Y_1(t),\dots,\bar Y_{n-1}(t)$ 
of $\bar\eta^\prime(t)^\perp$.

The Jacobi fields $\bar J_i$ can explicitly be computed as follows. We have
$\bar J_1(t)=\cosh(t)\bar Y_1(t)$, and 
$\bar J_i(t)=\sinh(t)\bar Y_i(t)$ for $i\geq 2$. 
In particular, if we denote by $\bar j(t)$ the determinant
of $\bar J(t)$ then 
\[\bar j(t)=\sinh^{n-2}(t)\cosh(t).\]
Thus $\bar j^\prime(t)=(n-2)\sinh^{n-3}(t)\cosh^2(t)+
\sinh^{n-1}(t)$ and hence
\[\bar j^\prime(t)=((n-2)\coth(t)+\tanh(t))\bar j(t).\]

Using the notations from the lemma and the text preceding it, 
Theorem 3.2, Theorem 3.4 and Section 6.1 of \cite{E87} show that 
\begin{equation}\label{eschenburg}
j^\prime(t)/j(t)\geq \bar j^\prime(t)/\bar j(t)\end{equation}
for all $t>0$.

Write $b(t)=\frac{1}{n-1}\sinh^{n-1}(t)$; then 
$b(t)=\int_0^t\bar j(s)ds$ and 
\begin{equation}\label{hyperbolic}
\frac{d}{dt}\log b(t)=\frac{b^\prime(t)}{b(t)}=(n-1)\coth(t)>n-1
\end{equation}
for all $t$. 

For $a(t)=\int_0^tj(s)ds$ we have $a^\prime(t)=j(t)$. 
By the estimate (\ref{hyperbolic}), it now suffices to show 
that 
\[\frac{d}{dt}\log a(t)=\frac{a^\prime(t)}{a(t)}\geq \frac{d}{dt}\log b(t)\]
for all $t>0$. 

Let $\delta>0$. It suffices to show that 
for all $t>0$ we have 
$\frac{d}{dt}\log a(t)\geq (1-\delta)\frac{d}{dt}\log b(t)$. 
To this end note that by comparison, the inequality holds  true for all small $t$ (see \cite{E87} for 
details). 
As this condition is closed, if the inequality does not hold for all $t$ then 
there is a number $T_0>0$ so that the 
estimate holds true for $t\leq T_0$, and it is violated for 
$T_0<t<T_0+\tau$ where $\tau >0$. 
Then we have 
$\frac{d}{dt}\log a(T_0)=(1-\delta)\frac{d}{dt}\log b(T_0)$,

Now 
\[\frac{d^2}{dt^2}\log a=\frac{d}{dt}\frac{a^\prime}{a}=\frac{a^{\prime\prime}a-(a^\prime)^2}{a^2}
=\frac{a^{\prime\prime}}{a}-(\frac{a^\prime}{a})^2=\frac{a^\prime}{a}(\frac{a^{\prime\prime}}{a^\prime}-
\frac{a^\prime}{a}).\]
Inequality (\ref{eschenburg}) implies that 
$\frac{a^{\prime\prime}}{a^\prime}(T_0)\geq 
\frac{b^{\prime\prime}}{b^\prime}(T_0)$.
As $\frac{a^\prime}{a}(T_0)=(1-\delta)\frac{b^\prime}{b}(T_0)$,
we conclude that
\[\frac{d^2}{dt^2}\log a(T_0)> (1-\delta)
\frac{d^2}{dt^2}\log b(T_0).\]
Using Taylor expansion, we deduce that
$\frac{a^\prime}{a}(s)\geq (1-\delta)\frac{b^\prime}{b}(s)$ for all $s>T_0$ 
which are sufficiently close to $T_0$. This  
contradicts the choice of $T_0$. 
Since $\delta >0$ was arbitrary, the lemma follows.
\end{proof}

A \emph{cusp} $T$ is an unbounded component of $M_{\rm thin}$. 
It is homeomorphic to $N\times [0,\infty)$ where
$N$ is a closed manifold of dimension $n-1$. 
The manifold $N$ is 
homeomorphic to the quotient of 
a horosphere in the universal covering $\tilde M$ of $M$ 
by a parabolic subgroup of the isometry group of 
$\tilde M$. 
As before, the boundary of $T$ need not be smooth, but
everything said so far for boundaries of Margulis tubes
is also valid for cusps (see the remark after Theorem 2.14 in 
\cite{BCD93}). In particular, 
Theorem \ref{thmbcd93} holds true for boundaries of cusps.

A version of Lemma \ref{volumegrowth} is also valid for cusps and follows with exactly the same
argument. Namely, let $T$ be a cusp, with boundary $\partial T$. Write $T=\partial T\times [0,\infty)$ where
$\partial T\times \{s\}$ is the hypersurface of distance $s$ to $\partial T$. 
Let $\eta:[0,\infty)\to T$ be a radial
geodesic and let $J(t)=(J_1(t),\dots,J_{n-1}(t))$ be Jacobi fields with the following properties. 
The vectors $J_1(0),\dots,J_{n-1}(0)$ define an orthonormal basis of 
$\eta^\prime(0)^\perp$, and $\Vert J(t)\Vert \to 0$ $(t\to \infty)$. Denote by $j(t)$ the 
determinant of $J(t)$, viewed as a matrix with respect to 
an orthonormal basis of $\eta^\prime(t)^\perp$ defined by parallel vector fields along $\eta$.

 Let $\bar j(t)$ be the corresponding function for Jacobi fields defined by a horoball in hyperbolic 
 $n$-space. Explicit calcuation yields $\bar j(t)=e^{-(n-1)t}$ and hence
 $\bar j(T)= (n-1)\int_T^\infty \bar j(s)ds$ for all $T$. 
 The comparison argument from the proof of 
 Lemma \ref{volumegrowth} together with the results in \cite{E87}  
 imply that the inequality $j(T)\geq \int_T^\infty j(s)ds$ holds true for cusps in 
 a manifold $M$ of curvature $\leq -1$. In fact, this statement can also be obtained as a limiting
 case  of Lemma \ref{volumegrowth} by reparametrizing 
the radial geodesic arc $\eta$ of the Margulis tube, 
choosing $\eta(R)$ as a basepoint, renormalizing 
the Jacobi fields with rescaling and the Gram-Schmidt procedure 
and letting $R$ tend to infinity.

As one consequence of the above discussion, 
by possibly replacing $M_{\rm thick}$ by 
a neighborhood of uniformly bounded radius
we may assume that 
\begin{equation}\label{volume}
\frac{3}{2} {\rm vol}(M_{\rm thick})\geq  {\rm vol}(M).\end{equation}

Consider again a Margulis tube $T$ in $M$.
Recall that the boundary
$\partial T$ of $T$ equals the set $\{\rho=R\}$ which can 
be viewed as a graph over the unit normal bundle $N(\gamma)$ of $\gamma$
for some smooth function $R:N(\gamma)\to (0,\infty)$.
Here we use Theorem \ref{thmbcd93} to assure that the function $R$ is smooth. 
We equip $\partial T$ with the volume element determined by this description
(this is in general not the volume element of $\partial T$ viewed as
a smooth hypersurface in $M$). By this we mean that the volume 
element of the hypersurface $\partial T$ at a point 
$(\sigma_0,s_0,R(\sigma_0,s_0))$ equals the radial 
projection of the volume 
element of the local hypersurface $\{(\sigma,s,R(\sigma_0,s_0))\}$ 
passing through $(\sigma_0,s_0,R(\sigma_0,s_0))$.  
This makes sense since by  Theorem \ref{thmbcd93}, 
radial geodesics intersect $\partial T$ transversely. 
Or, equivalently, 
this volume element is chosen in such a way that the Jacobian at 
$(\sigma,s)$ of the map 
$(\sigma,s)\to (\sigma,s,R(s,\sigma))$  equals
the Jacobian of the normal exponential map $(\sigma,s)\to 
\exp(R(\sigma,s)(\sigma,s))$, and this is just the function $j$ 
from Lemma \ref{volumegrowth}. 
The same construction is also valid for a cusp $T$, 
and we equip $\partial T$ with the corresponding volume element
obtained by projection of the volume element on local 
hypersurfaces orthogonal to the radial geodesics.

\begin{lemma}\label{partialintegration}
Let $T\subset M_{\rm thin}$ be a Margulis tube or a cusp with boundary 
$\partial T$, and 
let $f$ be a smooth function
on $T$ with $\int_T f^2\geq \int_{\partial T}f^2$; then 
\[\int_T f^2\leq \frac{4}{(n-2)^2}\int_T \Vert \nabla f \Vert^2.\]
\end{lemma}
\begin{proof} We begin with the case that $T$ is a Margulis tube. 
Let $\gamma:[0,\ell)\to M$ be a parameterization of the core geodesic 
of $T$ by arc length. We use normal exponential coordinates and Lemma \ref{volumegrowth}.
Let $\rho$ be the radial distance from $\gamma$ and let 
$j(\sigma,s,\rho)$ be the Jacobian of the normal exponential map 
at the point with
"coordinates" $(\sigma,s,\rho)$. 
Then we have 
\[\int_Tf^2=\int_{S^{n-2}}d\sigma\int_0^\ell ds
\int_0^{R(\sigma,s)} f^2 j(\sigma,s,\rho) d\rho.\]

Define $a(\sigma,s,\rho)=\int_0^{\rho}j(\sigma,s,u)du$. 
By Lemma \ref{volumegrowth} 
and the definition of the volume element on $\partial T$, 
integration by parts along the radial rays from $\gamma$ 
yields 
\begin{equation}\label{inequality1}
\int_T f^2\leq \frac{1}{n-1}
\int_{\partial T} f^2 - 
2\int_{S^{n-2}}d\sigma\int_0^\ell ds \int_0^{R(\sigma,s)} ff^\prime 
a(\sigma,s,\rho) d\rho\notag \end{equation}
where $f^\prime$ is the derivative of $f$ in direction of the 
radial variable $\rho$.

By the assumption in the lemma, 
we have 
\[\int_{\partial T} f^2 \leq \int_Tf^2\]
and therefore taking absolute values 
and using Lemma \ref{volumegrowth} once more, 
we obtain
\begin{align}
\frac{n-2}{n-1}\int_T f^2  & \leq  
\frac{2}{n-1}
\int_{S^{n-2}}d\sigma\int_0^\ell ds \int_0^R \vert  ff^\prime \vert j(\sigma,s,\rho) d\rho
\notag\\ & =
\frac{2}{n-1}\int_T \vert  ff^\prime \vert 
\leq \frac{2}{n-1}
(\int_T f^2)^{1/2}(\int_T \Vert \nabla f\Vert^2)^{1/2} \notag\end{align}
where the last inequality follows from 
Schwarz's inequality and from $\vert f^\prime\vert \leq 
\Vert \nabla f\Vert$ (compare the proof of  
Lemma 1 of \cite{DR86}). Dividing by 
$\frac{2}{n-1}(\int_T f^2)^{1/2}$ and 
squaring the resulting inequality yields the lemma in the case that
$T$ is a Margulis tube.

The argument for a cusp is almost identical.
Thus let $T$ be such a cusp, with boundary $\partial T$. 
Then $\partial T$ is an closed manifold, 
and $T$ is diffeomorphic to $\partial T\times [0,\infty)$.
Let $\rho:T\to [0,\infty)$ be the radial distance from $\partial T$.

Let 
$f:T\to \mathbb{R}$ be a smooth square integrable function with 
$\int_Tf^2\geq \int_{\partial T}f^2$ and let 
$j:\partial T\times [0,\infty)\to 
(0,\infty)$ be the function which describes the volume form on $T$
as discussed above. The version of 
Lemma 2.1 for cusps and integration by parts is used as in the proof for
Margulis tubes.
Since $f$ is square integrable,
as $R\to \infty$ we have $\int_{\rho=R}f^2\to 0$ and hence
the same calculation as before yields 
(here $d\mu$ is the 
volume element on $\partial T$
 used in the lemma) 
\begin{align} \int_Tf^2 &=\int_{\partial T}d\mu
\int_{0}^\infty f^2 j(x,\rho)d\rho
=\lim_{R\to\infty}
\int_{\partial T}d\mu\int_{0}^R f^2j(x,\rho)d(\rho)\notag \\
&\leq \frac{1}{n-1}\bigl(
\int_{\partial T}f^2+2\int_{\partial T}d\mu
\int_0^\infty ff^\prime j(x,\rho) d\rho \bigr).\notag
\end{align}  
As before,   
this yields the required estimate. 
Note that the change of sign in this formula stems from the fact that
here $\rho$ is the radial distance from the boundary, while for Margulis
tubes, $\rho$ denoted the radial distance from the core geodesic which 
seems more natural in that context. 
\end{proof}

For a Margulis tube $T$ or a cusp, let $\hat T$ be the set of all points 
$x \in T$ whose radial distance from the boundary $\partial T$ is at least one. 
Thus if $T$ is Margulis tube, determined by a closed
geodesic and a function $R(\sigma,s)>0$ on the normal bundle of that geodesic, then 
using "coordinates" $(\sigma,s,\rho)$ as before,  
we have  $\hat T=\{(\sigma,s,\rho)\in T\mid \rho\leq R(\sigma,s)-1\}$. 
By our assumption on Margulis tubes,
$T-\hat T$ is diffeomorphic to $\partial T\times [0,1)$.
Write $\hat M_{\rm thick}=M-\cup \hat T$ where the union is over all 
Margulis tubes and cusps.

For a smooth square integrable function $f$ on $M$ denote by
\[{\cal R}(f)=\int_M \Vert \nabla f\Vert^2/\int_M f^2\]
the Rayleigh quotient of $f$.

\begin{lemma}\label{rayleigh}
Let $f:M\to \mathbb{R}$
be a smooth square integrable function with
Rayleigh quotient ${\cal R}(f)<  \frac{(n-2)^2}{12}$; then 
\[\int_{\hat M_{\rm thick}}f^2\geq \frac{1}{3}\int_M f^2.\]
\end{lemma}
\begin{proof}
By our assumption on the components of the thin part of $M$, the set
\[A=\hat M_{\rm thick}-M_{\rm thick}\]  is a union of \emph{shells},
i.e. submanifolds of $M$ with boundary which either are
diffeomorphic to $S^{n-2}\times S^1\times [0,1]$
(if the component is a Margulis tube) or to 
$N\times [0,1]$ (if the component is a cusp),
where
$N$ is a quotient of a horosphere by a rank
$n-1$ parabolic subgroup of the isometry group of 
the universal covering $\tilde M$ of $M$.

Let $f:M\to \mathbb{R}$ be a smooth function with 
${\cal R}(f)<\frac{(n-2)^2}{12}$. We want to show that
\begin{equation}\label{integral}
\int_{\hat M_{\rm thick}}f^2\geq 
\frac{1}{3}\int_M f^2,\notag\end{equation} 
and to this end we assume to the contrary 
that $\int_{\hat M_{\rm thick}}f^2<
\frac{1}{3}\int_M f^2$.

Recall that $M_{\rm thin}$ is a disjoint union 
of a finite number of 
Margulis tubes and cusps, say 
$M_{\rm thin}=\cup_{i=1}^kT_i$.
For each of these tubes and cusps $T_i$, let $r_i$ be the radial distance
function to the boundary hypersurface, i.e. $r_i(x)$ is 
the length of the radial arc
connecting the point $x\in T$ to $\partial T$. By reordering we may assume that 
there exists a number $p\leq k$ such that 
for all $i\leq p$, there is some 
$s_i\leq 1$ such that
\[\int_{\{r_i=s_i\}\cap T_i}f^2\leq \int_{T_i\cap \{r_i\geq s_i\}}f^2\]
and that for $i>p$, such an $s_i$ does not exist.
Here we use the volume element on the hypersurfaces
$\{r_i=s_i\}$ as in Lemma \ref{partialintegration}. 

We distinguish two cases. In the first case,
$\sum_{i=1}^p\int_{T_i-A}f^2\geq \frac{1}{3}\int_Mf^2$.
Lemma \ref{partialintegration} then shows that
\begin{align}
\int_M\Vert \nabla f\Vert^2\geq 
\sum_{i=1}^p\int_{T_i\cap \{r_i\geq s_i\}}\Vert \nabla f\Vert^2 \notag \\
\geq \frac{(n-2)^2}{4}
\sum_{i=1}^p\int_{T_i\cap \{r_i\geq s_i\}}f^2\geq 
\frac{(n-2)^2}{12}\int_M f^2.\notag\end{align}
Thus we have
${\mathcal R}(f)\geq \frac{(n-2)^2}{12}$ 
which contradicts our assumption on $f$.

In the second case, we have
$\sum_{i=1}^p\int_{T_i-A}f^2< \frac{1}{3}\int_Mf^2$.
Then 
\begin{equation}\label{eqlemma42}
\sum_{i=p+1}^k\int_{T_i-A}f^2\geq \frac{1}{3}\int_M f^2.
\end{equation} 
But for each $i>p$, integration of the defining equation
\begin{equation}\label{integral2}
\int_{T_i\cap \{r_i=s\}}f^2\geq 
\int_{T_i\cap \{r_i\geq s\}}f^2\end{equation}
over the shell $T_i\cap A=\{0\leq r_i\leq 1\}$ yields
\[\int_0^1ds \int_{T_i\cap \{r_i=s\}}f^2=
 \int_{T_i\cap A}f^2\geq \int_{T_i-A}f^2.\]
Summing over $i\geq p+1$ and using inequality
(\ref{eqlemma42}), we obtain
\[\int_{\cup_{i=p+1}^kT_i\cap A}f^2\geq \sum_{i=p+1}^k\int_{T_i-A}f^2
\geq \frac{1}{3}\int f^2.\]
As $A\subset \hat M_{\rm thick}$, this 
contradicts the assumption on $f$. The
lemma follows. 
\end{proof}

The next proposition completes the proof of Theorem \ref{thm1} for a choice of a Margulis
constant so that $\hat M_{\rm thick}$ as defined above is contained in the
thick part of $M$ for this constant.

\begin{proposition}\label{lowerbound}
For all $k\geq 0$ we have 
\[\lambda_k(M)\geq \min\{
\frac{1}{3}\lambda_k(\hat M_{\rm thick}),
\frac{(n-2)^2}{12}\}\]
for any finite volume oriented Riemannian mani\-fold $M$ of dimension
$n\geq 3$ and curvature $\kappa\in [-b^2 -1]$.  
\end{proposition}
\begin{proof} Let $M$ be a finite volume oriented Riemannian manifold of dimension
$n\geq 3$ and curvature $\kappa \leq -1$.
We use the previous notations.

If $M$ is non-compact, then by Corollary 3.2 of \cite{H04}, 
the bottom of the essential 
spectrum of $M$ is not smaller than $(n-1)^2/4$. Thus the intersection of the spectrum of 
$M$ with the interval $[0,\frac{(n-2)^2}{12}]$ consists of a finite number of eigenvalues.

Let ${\cal H}(M)$ be the Sobolev space 
of square integrable functions on $M$, with square integrable weak 
derivatives. 
Let $k>0$ be such that $\lambda_{k}(M)<(n-2)^2/12$. 
Let $m\leq k-1$ be the largest number so that $\lambda_m(M)<\lambda_k(M)$.
Note that as we count eigenvalues with multiplicities,
we may have $m<k-1$. Choose a $k-m$-dimensional subspace 
$E_k\subset {\cal H}(M)$ of the eigenspace for the eigenvalue $\lambda_k(M)$ and define 
$E=V\oplus E_k$ where $V$ is the direct sum
of the eigenspaces for eigenvalues strictly smaller than $\lambda_k(M)$.
In particular, 
$E$ is a $k+1$-dimensional linear subspace of the Hilbert space ${\cal H}(M)$.

By construction, $\hat M_{\rm thick}$ is a smooth manifold with smooth boundary. 
Denote by ${\cal H}(\hat M_{\rm thick})$ the Sobolev space 
of square integrable functions
on $\hat M_{\rm thick}$ with square integrable weak derivatives. 
Here the weak derivative of a function $f$ on $\hat M_{\rm thick}$
is a vector field $Y$ so that
\[\int_{\hat M_{\rm thick}} \langle Y,X\rangle=
-\int_{\hat M_{\rm thick}}f{\rm div}(X)\]
for all smooth vector fields $X$ on $\hat M_{\rm thick}$ with compact
support in the interior of $\hat M_{\rm thick}$.
Green's formula implies that
${\cal H}(\hat M_{\rm thick})$ contains all functions 
$f$ on $\hat M_{\rm thick}$ which are smooth up to and including the 
boundary.

As smooth functions are dense in ${\cal H}(M)$ and 
${\cal H}(\hat M_{\rm thick})$, there 
is a natural linear one-Lipschitz restriction map
$\Pi:{\cal H}(M)\to {\cal H}(\hat M_{\rm thick})$. We denote by 
$W$ the image of the linear subspace $E$ under $\Pi$. 

We claim that
${\rm dim}(W)=k+1$. To this end assume otherwise. Then there is 
a normalized 
function $f\in E$ (i.e. $\int_M f^2=1$) 
whose restriction to $\hat M_{\rm thick}$ vanishes.
But by the definition of $E$, the Rayleigh quotient 
${\cal R}(f)$ of $f$ is at most $\lambda_k(M)<(n-2)^2/12 $ and therefore 
Lemma \ref{rayleigh} shows that $\int_{\hat M_{\rm thick}}f^2\geq \frac{1}{3}\int_M f^2$.

As ${\rm dim}(W)=k+1$ and as 
${\cal H}(\hat M_{\rm thick})$ is the function space for 
$\hat M_{\rm thick}$ with Neumann boundary conditions (see p.14-17 in \cite{C84}),  
Rayleigh`s theorem shows that 
there exists a normalized function $f\in E$ 
with ${\cal R}(\Pi(f))={\cal R}(f\vert \hat M_{\rm thick})
\geq\lambda_k(\hat M_{\rm thick})$.
Furthermore, as $f\in E$, we have 
${\cal R}(f)\leq \lambda_k(M)$. Note that $f$ and $\Pi(f)$ are smooth. 

By Lemma \ref{rayleigh}, we have
$\int_{\hat M_{\rm thick}}f^2\geq \frac{1}{3}\int_M f^2=\frac{1}{3}$ and hence
\[\lambda_k(M)\geq {\cal R}(f)\geq \int_{\hat M_{\rm thick}}\Vert \nabla f\Vert^2 
\geq \frac{1}{3}{\cal R}(f\vert_{\hat M_{\rm thick}})\geq 
\frac{1}{3}\lambda_k(\hat M_{\rm thick}).\] 
This is what we wanted to show.
\end{proof}

As an easy consequence, we obtain the estimate of Schoen \cite{S82}.

\begin{corollary}\label{eigenbelow}
For every $n\geq 3$ and every $b\geq 1$ 
there exists a number $\chi(n,b)>0$ such that 
\[\lambda_1(M)\geq \frac{\chi(n,b)}{{\rm vol}(M)^2}\]
for every finite volume Riemannian manifold $M$ of 
dimension $n$ and curvature
$\kappa\in [-b^2,-1]$. 
\end{corollary}
\begin{proof} Let $\epsilon >0$ be a Margulis constant for Riemannian 
manifolds $M$ of dimension $n\geq 3$ and curvature in the interval
$[-b^2,-1]$. As $M_{\rm thick}\not=\emptyset$, the manifold $M$ contains
an embedded ball of radius $\epsilon$ which is isometric to a ball of the
same radius in a simply connected manifold of curvature contained in 
$[-b^2,-1]$. By comparison, the volume of such a ball is bounded from 
below by a universal constant $a(n,b)$ only depending on $n$ and $b$ and hence
the volume of every Riemannian $n$-manifold of curvature in $[-b^2,-1]$ is bounded 
from below by 
$a(n,b)$. 
Thus Theorem \ref{thm1} shows that 
$\lambda_1(M)\geq \min\{a(n,b)^2(n-2)^2/12{\rm vol}(M)^2,
\lambda_1(M_{\rm thick})/3\}$.

The manifold with boundary $M_{\rm thick}$ 
is uniformly quasi-isometric to a finite connected graph $G$,
and its first nontrivial eigenvalue 
$\lambda_1(M_{\rm thick})$ with Neumann boundary conditions
is uniformly equivalent to the first nontrivial eigenvalue of the 
graph Laplacian \cite{Man05} (see also Lemma 2.1 of \cite{LS12} for an explicit
formulation of this fact).

Now the first eigenvalue $\lambda_1(G)$ 
of a finite graph $G$ satisfies $\lambda_1(G)\geq h^2(G)/2$ where
$h(G)$ is the so-called \emph{Cheeger constant} of the graph \cite{Chu96}.
This Cheeger constant is defined to be 
\[\min\frac{\vert E(U,U^\prime)\vert}{\min\{\vert U\vert,\vert U^\prime\vert\}}\]
where $U$ is a subset of the vertex set ${\mathcal V}(G)$ of $G$, 
$U^\prime={\mathcal V}(G)-U$ and where $E(U,U^\prime)$ is the set of 
all edges which connect a vertex in $U$ to a vertex in $U^\prime$. 
As $G$ is connected, 
this Cheeger constant is at least $1/\vert {\cal V}(G)\vert \sim
1/{\rm vol}(M)$ which yields the corollary. 
\end{proof}

\section{Bounding small eigenvalues from above}\label{above}

In this section we restrict to the investigation of orien\-ted 
hy\-per\-bolic 
3-manifolds of finite volume. 
Our goal is to prove Theorem \ref{thm2}.

We begin with analyzing in more detail functions on Margulis tubes 
in such a hyperbolic $3$-manifold $M$. 
For a sufficiently
small Margulis constant
$\epsilon>0$, the boundary $\partial T$ of each such tube $T$ is a flat
torus whose injectivity radius for the induced metric 
roughly equals $\epsilon$ \cite{BCD93}. 
Furthermore, radial geodesics intersect $\partial T$ orthogonally. 
Note that an analogous statement is
not true in higher dimensions.

The \emph{radius} ${\rm rad}(T)$ 
of the tube $T$ is the distance of $\partial T$ to 
the core geodesic $\gamma$. The following lemma is only valid
in dimension three.

\begin{lemma}\label{radiusbound}
There is a number $q_1=q_1(\epsilon)>0$ only depending on $\epsilon$ such that
${\rm rad}(T)\geq \log {\rm vol}(\partial T)-q_1$. 
\end{lemma}
\begin{proof} 
Write $R={\rm rad}(T)$. We may assume that $R>1$.
The tube $T$ is isometric to the quotient of the tubular
neighborhood $N(\tilde \gamma,R)$ of radius $R$ 
of a geodesic $\tilde \gamma$ in $\mathbb{H}^3$ under
an infinite cyclic group of loxodromic isometries. Up to conjugation, the generator
$\psi$ of this group 
is determined by its complex translation length $\chi\in \mathbb{C}$
with $\Re(\chi)=\ell<2\epsilon$. Here $\ell$ is the length of the core geodesic of $T$.

The boundary $\partial N(\tilde \gamma, R)$ of $N(\tilde \gamma,R)$ 
is a flat two-sided infinite cylinder of circumference $2\pi \sinh R$. 
For a fixed identification of the fibre
of the unit normal circle bundle of $\tilde \gamma$ over the point
$\tilde \gamma(0)$ with the unit circle $S^1$, parallel transport along $\tilde\gamma$ and the 
normal exponential map determine global ``coordinates'' on 
$\partial N(\tilde \gamma,R)$. These ``coordinates'' consist in a 
diffeomorphism 
$S^1\times \mathbb{R}\to \partial N(\tilde \gamma,R)$. The isometry $\psi$   
identifies the meridian
$\exp(R S^1\times \{0\})$ 
with the meridian $\exp(R S^1\times \{\ell\})$   
by an isometry
which is given by rotation with angle $\Im \chi$
in these coordinates. In particular, 
the volume of the boundary torus $\partial T$ of $T$ 
equals 
\begin{equation}\label{volumeformula}
{\rm vol}(\partial T)=2\pi\ell \sinh R\cosh R
\end{equation} 
and is independent of 
$\Im \chi$. 

The second fundamental form of the torus 
$\partial T$ is uniformly bounded, independent of $R>1$. 
This implies  that 
the length of
the shortest geodesic on $\partial T$ is contained
in an interval $[\epsilon,a\epsilon]$ where $a>0$ is a
universal constant (compare Theorem 2.14  and the 
following discussion in \cite{BCD93}).
As we are only interested in tubes of large radius, 
we may assume that the length
$2\pi \sinh R$ of the meridian of 
$\partial T$ is bigger than $a\epsilon$. 

Now the distance on $\partial N(\tilde \gamma,R)$ 
for the intrinsic path metric between the 
circles corresponding to the coordinates $s=0,s=\ell$ 
equals $\ell\cosh R$. Here as before, $s$ is the length parameter
of $\gamma$. A shortest closed geodesic
on the flat torus 
$\partial T$ lifts to a straight line segment on
$\partial N(\tilde \gamma,R)$. Since $2\pi\sinh R>a\epsilon$, 
this line segment connects the circle $\{s=0\}$ to the circle
$\{s=k \ell\}$ for some integer $k\geq 1$. From this we conclude that
\[\ell\cosh R\leq a\epsilon.\]

From the formula (\ref{volumeformula})
for ${\rm vol}(\partial T)$, we deduce that
$\sinh R\geq \frac{1}{2\pi a\epsilon}{\rm vol}(\partial T)$ and hence 
\[e^R\geq \frac{1}{\pi a\epsilon}{\rm vol}(\partial T)\]
which is what we wanted to show.
\end{proof}

\begin{remark} 
For hyperbolic manifolds of dimension $n\geq 4$, 
Proposition 2 of \cite{BS87} states a reverse inequality:
Namely, if $U$ is any Margulis tube, and if $r(U)$ is
the largest distance of a point in $U$ to the boundary,
then ${\rm Vol}(U)\geq d_n\sinh(\frac{1}{3}r(U))$
where $d_n>0$ is a constant only depending on $n$.
The case ${\rm dim}(M)=3$ is very special as every point
on the boundary of a Margulis tube has the same distance
to the core geodesic. We refer to p.3 of \cite{BCD93} for more
information.
\end{remark} 

In the statement of the next proposition, we use the fact that 
given a Margulis tube $T$ of radius at least three, the volume of 
$T$ is uniformly proportional to the volume of its boundary 
$\partial T$.

\begin{proposition}\label{tube}
There exists a number $q_2=q_2(\epsilon)>0$ with the following
property. Let $T\subset M$ be a Margulis tube or a cusp 
with boundary $\partial T$.
Let $f:\partial T\to \mathbb{R}$ be a function 
(not neccessarily of zero mean) whose
Rayleigh quotient equals $d\geq 0$. 
Then there is an extension of $f$ to a smooth function $F$ on $T$ with 
the following properties.
\begin{enumerate}
\item $\frac{1}{4}\int_{\partial T}f^2\leq 
\int_TF^2\leq \frac{1}{2}\int_{\partial T}f^2.$
\item The 
Rayleigh quotient of $F$ is at most
$dq_2\log {\rm vol}(T)$. 
\item If $\int_{\partial T}f=0$ then $\int_T F=0$.
\end{enumerate}
\end{proposition}
\begin{proof} Let $\gamma$ be the core curve of the tube $T$.
Use the coordinates 
$(\sigma,s,\rho)$ on $T$, given by the angular coordinate
$\sigma$ on the unit normal circle over a point in $\gamma$, the length
parameter $s$ on $\gamma$ and the distance $\rho$ from $\gamma$. 
If $R>0$ is the radius of $T$ then 
the boundary $\partial T$ of $T$ is the surface
$\rho=R$. This boundary is a flat torus whose injectivity
radius equals $\epsilon$ up to a universal multiplicative
constant. 

By Lemma \ref{radiusbound}, 
the radius $R$ of the tube satisfies
\[-\theta=\log{\rm vol}(\partial T)  -q_1-1-R\leq -1\] where 
$q_1>0$ is a universal constant.  By the inequality
(\ref{mincontrol}) in Section 2, we 
may assume that 
$R-\theta \geq 3$.

Let $f:\partial T=\{\rho=R\}\to \mathbb{R}$ be a smooth function
with Rayleigh quotient $d\geq 0$. 
Decompose $f=f_0+g$ 
where $f_0$ is a constant function and $\int_{\partial T} g=0$. 
Extend $f_0$ to a constant
function $F_0$ on $T$, and extend the function $g$  
to a function $G:T\to \mathbb{R}$ by
\begin{equation}
G(\sigma,s,\rho)=
\begin{cases}
g(\sigma,s,R) &\text{if $\theta+1\leq \rho\leq R$;}\\
(\rho-\theta) g(\sigma,s,R) &\text{if 
$\theta\leq \rho\leq \theta +1$;}\\
0 &\text{otherwise.}
\end{cases}
\end{equation}
Then $F=F_0+G$ is continuous, smooth away from the hypersurfaces
$\rho=\theta+1$ and $\rho=\theta$, with uniformly bounded derivative.
Standard Sobolev theory then implies that $F$ is 
contained in the Sobolev space
of square integrable functions with square integrable weak derivative,
and $F\vert \partial T=f$.

As for each $r\in (0, R)$ 
the radial projection $(\sigma,s,r)\to (\sigma,s,R)$ 
of the torus $\{\rho =r\}$ onto the boundary torus $\partial T$ 
is a homothety, with dilation
$\sinh(R)\cosh(R)/\sinh(r)\cosh(r)$, we have $\int_TG =0$. 
Furthermore, 
integration using the explicit description of the metric on $T$
as a warped product metric (see Section \ref{below} for details) yields that
\begin{equation}\label{squarenorm}
\int_T (F_0+G)^2\in [\frac{1}{4}\int_{\partial T} (f+g)^2,\frac{1}{2}\int_{\partial T}(f+g)^2].
\end{equation}
Namely, as $(F_0+G)^2(\sigma,s,\rho)\leq f^2(\sigma,s,R)$ for all $\rho$, 
the upper bound follows from
\[\int_T(F_0+G)^2\leq \int_{\partial T}f^2\int_\theta^R\frac{\sinh \rho \cosh \rho}{\sinh R\cosh R}d\rho
=\frac{\sinh^2(R)-\sinh^2(\theta)}{2\sinh R\cosh R}\int_{\partial T}f^2.\]
To establish the lower bound, simply note that
\begin{align}
\int_T(F_0+G)^2\geq \int_{\partial T}f^2\int_{\theta+1}^R 
\frac{\sinh \rho \cosh \rho}{\sinh R\cosh R}d\rho\notag\\
=\frac{\sinh^2(R)-\sinh^2 (\theta+1)}{2\sinh R\cosh R}\int_{\partial T}f^2.
\notag\end{align}
Up to adjusting the requirement on the radius of the Margulis tube, we may
assume that $\sinh^2(\theta +1)\leq \frac{1}{4} \sinh(R)^2$ and
$\frac{3}{4}\sinh(R)\geq \frac{1}{2}\cosh(R)$
which then results in the 
estimate stated in the first part of the proposition. 

Now $\int_T\Vert \nabla(F_0+G)\Vert^2=\int_T\Vert \nabla G\Vert^2$
and therefore for the second part of the statement in the proposition, 
it suffices to show that
\begin{equation}\label{logestim}
  \int_T\Vert \nabla G\Vert^2\leq c_1(R-\theta) 
  \int_{\partial T}\Vert \nabla g\Vert^2\end{equation}
for a universal 
constant $c_1>0$. Namely, the discussion in the previous paragraph shows
that the integral of the constant function one over the shell
$\theta\leq \rho\leq R$ is bounded from below by 
$\frac{1}{4}{\rm vol}(\partial T)$. In particular, we have 
$\log{\rm vol}(T)\geq \log {\rm vol}(\partial T)-\log 4$. 
On the other hand, by the choice of 
$\theta$, we also have $R-\theta=\log{\rm vol}(\partial T)-q_1-1$. 

To establish the estimate (\ref{logestim})
recall from Chapter 2 of \cite{C84} that the first 
nonzero eigenvalue of $\partial T$ is not smaller than
$c_2/{\rm vol}(\partial T)^2$ where $c_2>0$ is a universal constant
(here we use that the injectivity radius of $\partial T$ is at least $\epsilon$).
In particular, by the definition of the number $\theta$, this eigenvalue is 
not smaller than $c_3e^{-2(R-\theta)}$ where $c_3>0$ is a universal constant.
As a consequence, we have 
\begin{equation}\label{rayleighbd}
\int_{\partial T}\Vert \nabla g\Vert^2\geq c_3e^{-2(R-\theta)}\int_{\partial T}g^2.
\end{equation}

Now for $r\in [\theta+1,R]$, the radial projection of 
the torus $\{\rho=r\}$ onto $\partial T$ 
scales the metric with a fixed constant. Moreover,
the directional derivative of $G$ in direction of the
radial vector field vanishes. This implies that   
\[\int_{\rho=r}\Vert \nabla G(\sigma,s,r)\Vert^2 =
\int_{\partial T} \Vert \nabla G(\sigma,s,R)\Vert^2\]
for $r\in [\theta +1,R]$ and hence
\[\int_{\theta+1\leq \rho\leq R} \Vert \nabla G(s,t,\rho)\Vert^2=
(R-\theta-1)\int_{\partial T}\Vert \nabla g\Vert^2.\]
On the other  hand, if
$\theta< r<\theta+1$ then 
\[\int_{\rho=r}\Vert \nabla G(s,t,\rho)\Vert^2 \leq 
\int_{\partial T}\Vert \nabla G(s,t,R)\Vert^2+\int_{\rho=r}G(s,t,\rho)^2.\]
Here the second term in this inequality is the contribution of the
derivative of $G$ in direction of the radial vector field, and we use
that this vector field is normal to the hypersurfaces
$\{\rho=r\}$. 

There is a universal constant $c_4>0$ such that for each $r\in 
[\theta,\theta+1]$ the volume of the level surface
$\{\rho=r\}$  is not bigger than 
$c_4e^{-2(R-\theta)}({\rm vol}(\partial T))$. Integration of this
inequality over the interval $[\theta,\theta +1]$ yields
$\int_{\theta\leq \rho\leq \theta +1}G^2\leq
c_4e^{-2(R-\theta)}\int_{\partial T}g^2$. Hence by the estimate
(\ref{rayleighbd}), 
\[\int_{\theta\leq \rho\leq \theta+1}G^2\leq 
\frac{c_4}{c_3}\int_{\partial T}\Vert \nabla g\Vert^2.\]
Together this
implies
\begin{equation}\label{secondbound}
\int_T\Vert \nabla G\Vert^2\leq (R-\theta +\frac{c_4}{c_3})
\int_{\partial T} \Vert \nabla g\Vert^2. \end{equation}
The estimates (\ref{squarenorm},\ref{secondbound}) 
together with the choice of $\theta$ show 
part (1) and (2) of the proposition. The third part is immediate from the 
fact that by construction, if $\int _{\partial T}f=0$ then $F=G$ and 
$\int_{T}G=0$.

The case that $T$ is a cusp is completely analogous but easier 
and will be omitted.
\end{proof}

As an immediate consequence of Proposition \ref{tube} we obtain

\begin{corollary}\label{neumanntube}
The first nonzero eigenvalue with Neumann boundary conditions
of a three-dimensional hyperbolic Margulis tube or a cusp is at most
$q_3\log{\rm vol}(T)/{\rm vol}(\partial T)^2$
where $q_3>0$ is a universal constant. 
\end{corollary}
\begin{proof}
By the discussion in 
Chapter II, Section 2 of \cite{C84} 
and the fact that the injectivity radius of the flat torus
$\partial T$ is proportional to $\epsilon$,  
there exists a universal number $a>0$ and 
a function $f:\partial T\to \mathbb{R}$ with $\int_{\partial T}f=0$ and 
\[\int_{\partial T}\Vert \nabla f\Vert^2\leq a\int_{\partial T}f^2 /{\rm vol}(\partial T)^2.\]

By Proposition \ref{tube}, the function $f$ can be extended to a function
$F:T\to \mathbb{R}$ with $\int_T F=0$ and Rayleigh quotient
${\cal R}(f)\leq q_2a\log {\rm vol}(T)/{\rm vol}(\partial T)^2$. 
Now $F$ is smooth up to and including the boundary, and therefore 
by Rayleigh's principle as 
explained in Chapter 1 Section 5 of \cite{C84},
this implies 
that the first non-zero eigenvalue of $T$ with Neumann boundary conditions 
is bounded from above by $q_2a\log{\rm vol}(T)/{\rm vol}(\partial T)^2$ as
claimed in the corollary. 
\end{proof}


The strategy for the proof of 
Theorem \ref{thm2} from the introduction consists in extending
an eigenfunction $f$ on $\hat M_{\rm thick}$ to a function on $M$
using the construction in Proposition \ref{tube}. But the
control on square integrals and Rayleigh quotients
established in Proposition \ref{tube} depends on a control of
these data for the restriction of $f$ to the boundary of
$\hat M_{\rm thick}$, and there is no apparent relation to the
global square norm and the global Rayleigh quotient of $f$.
Note also that there is no useful Harnack inequality for the restriction of 
an eigenfunction on $\hat M_{\rm thick}$ to its boundary. 

To overcome this difficulty we establish some a-priori control
on the restriction of eigenfunctions for small eigenvalues to sufficiently large
tubular neighborhoods of the boundary of $\hat M_{\rm thick}$- the price
we have to pay is that we have to decrease the Margulis constant
which appears in Theorem \ref{thm2}. 

We begin with establishing 
a modified version of Lemma \ref{partialintegration}.
We only formulate this lemma for hyperbolic 3-manifolds although it 
is valid for arbitrary finite volume manifolds of curvature contained in 
$[-b^2,-1]$, where integration over radial distance hypersurfaces has to be
interpreted as in Section \ref{below}. 

Define a \emph{shell}
in a Margulis tube or cusp $T$ to be a subset of $T$ of the form
$\{s\leq \rho\leq t\}$ where $\rho$ is up to an additive  
constant the \emph{negative} of the
radial distance from the boundary of $T$
and $s<t$ (using the negative of the radial distance is for convenience here). 
The \emph{height} of the shell equals $t-s$. 
We always write a shell $N$ in the form $N=V\times [0,k]$ where
$k$ is the height of the shell, the manifold
$V$ is diffeomorphic to the boundary 
$\partial T$ of the tube or cusp, 
the real parameter $t$ is the radial distance
from the boundary component $V\times \{0\}$, and 
the boundary component $V\times \{k\}$ is closer to the boundary $\partial T$ of the tube or cusp. 
As before, we denote by ${\cal R}(f)$ the Rayleigh quotient of a function $f$.

\begin{lemma}\label{volumecontrol}
\begin{enumerate}
\item Let $k\geq 2$ and let 
$N=V\times [0,k]$ be a shell of height $k$ in 
a Margulis tube or cusp in a hyperbolic 3-manifold. 
Let $f:N\to \mathbb{R}$ be a function which is smooth up to
and including the boundary. 
If $\int_{V\times [0,1]}f^2\geq 3 \int_{V\times [k-1,k]}f^2$
then ${\cal R}(f)\geq \frac{1}{3}$.
\item Let $f$ be an eigenfunction on $\hat M_{\rm thick}$ with Neumann boundary
  conditions for an eigenvalue $\lambda<1$. Let $\tau>0$ be sufficiently
  small that
the closed tubular neighborhood of radius $\tau$ of 
the boundary $\partial \hat M_{\rm thick}$ in $\hat M_{\rm thick}$ is a
shell, diffeomorphic to $\partial \hat M_{\rm thick}\times [0,\tau]$.  
Then
$\int_{\partial \hat M_{\rm thick}\times \{\tau\}}f^2\geq 
\int_{\partial \hat M_{\rm thick}\times  \{0\}}f^2$. 
\end{enumerate}
\end{lemma}
\begin{proof} We begin with the proof of the first part of the lemma. Define 
  \begin{align}
    s&=\min\{t\in  [0,1]\mid \int_{V\times \{t\}}f^2
\geq \frac{1}{3}\int_{V\times [0,1]}f^2\}\text{ and }\notag\\
    u&=\max\{t\in [k-1,k]\mid
    \int_{V\times \{t\}}f^2\leq \int_{V\times [k-1,k]}f^2\}.\notag
    \end{align}
Then 
\begin{equation}
\int_{V\times [0,s]}f^2\leq \frac{1}{3}\int_{V\times [0,1]}f^2\leq 
\frac{1}{3}\int_N f^2,
\notag\end{equation}
in particular, by the assumption on $f$ in the first part
of the lemma, we have
\begin{align}
  \int_{V\times [s,u]}f^2& \geq \int_Nf^2-
\frac{1}{3}\int_{V\times [0,1]}f^2
  -\int_{V\times [k-1,k]}f^2\notag\\
&  \geq \int_Nf^2-
  \frac{2}{3}\int_{V\times [0,1]}f^2\geq \frac{1}{3}\int_Nf^2.
\notag\end{align}

Let us compute the derivative of the function
$a(t)=\int_{V\times \{t\}}f^2$. Recall that the radial projection 
$(x,t)\in V\times \{t\}\to (x,u)\in V\times \{u\}$ 
is a homothety which
scales the Lebesgue measure of the flat torus $\rho=t$ by 
a factor $b(t,u)$, with $\frac{d}{dt}b(t,u)\vert_{u=t}=c(t)\geq 2$
(in fact, for a cusp we have $c(t)\equiv 2$, and in the
case of a Margulis tube, if the radial distance
of $V\times \{t\}$ from the core geodesic equals $R$, then
$c(t)=(\cosh^2(R)+\sinh^2(R))/\sinh(R)\cosh(R)\geq 2$, compare the proof of Lemma
\ref{partialintegration}).
We obtain
\begin{equation}\label{areaderi}
  \frac{d}{dt}\int_{V\times \{t\}}f^2=c(t)\int_{V\times \{t\}}f^2
+2\int_{V\times \{t\}}f\nu(f)\end{equation}
where $\nu$ is the outer normal field of the hypersurface $V\times \{t\}$. 

Thus as $\int_{V\times \{u\}}f^2\leq \int_{V\times [k-1,k]}f^2
\leq \int_{V\times \{s\}}f^2$, we conclude that 
\begin{equation}\label{integral4}
  \int_{V\times \{u\}}f^2-\int_{V\times \{s\}}f^2=
\int_{V\times [s,u]} c(t)f^2+2f\nu(f)\leq 0\end{equation}
and hence 
\[\int_{V\times [s,u] }c(t)f^2\leq 2\vert \int_{V\times [s,u]} f\nu(f)\vert.\]
Using $c(t)\geq 2$ for all $t$ and $\nu(f)^2\leq \Vert \nabla f\Vert^2$ and
applying the Schwarz inequality as before,
we deduce
\begin{equation}\label{integral3} 
\int_{V\times [s,u]} f^2\leq \int_{V\times [s,u]}\Vert \nabla f\Vert^2\leq \int_N\Vert \nabla f\Vert^2.
\end{equation}
As
$\int_{V\times [s,u]}f^2\geq \frac{1}{3}\int_Nf^2$,
this provides the first statement in the lemma. 

To show the second statement in the lemma,
let $f$ be an eigenfunction on $\hat M_{\rm thick}$
with Neumann boundary conditions and eigenvalue $\lambda <1$.
Assume to the contrary that there exists a number 
$\tau>0$ such that
the tubular neighborhood of radius $\tau$ about the boundary
$\partial \hat M_{\rm thick}$ of $\hat M_{\rm thick}$ is a shell
diffeomorphic to $\partial \hat M_{\rm thick}\times [0,\tau]$
and that 
\begin{equation}\label{delta1}
\delta=\int_{\partial \hat M_{\rm thick}\times \{0\}}f^2-
\int_{\partial \hat M_{\rm thick}\times \{\tau\}}f^2>0.\end{equation}
Let $t_0\in (0,\tau]$ be the smallest number for which 
the equation (\ref{delta1}) holds true, with this number $\delta$. We then have
\[\frac{d}{dt}\int_{\partial \hat M_{\rm thick}\times \{t\}}f^2\vert_{t=t_0}\leq 0.\]
By formula (\ref{areaderi}), this
implies that $\int_{\partial \hat M_{\rm thick}\times \{t_0\}}
f\nu(f)<0.$

On the other hand, as $f$ is an eigenfunction for the
eigenvalue $\lambda$, with Neumann boundary conditions,
Green's formula yields that 
\begin{align}
\int_{\partial \hat M_{\rm thick}\times [0,t_0]}f\Delta f=
-\lambda \int_{\partial \hat M_{\rm thick}\times [0,t_0]}f^2\notag \\=
\int_{\partial \hat M_{\rm thick}\times \{t_0\}}f\nu(f)-
\int_{\partial \hat M_{\rm thick}\times [0,t_0]}\Vert \nabla f\Vert^2.\notag
\end{align}
Thus $\int_{\partial \hat M_{\rm thick}\times [0,t_0]}\Vert \nabla f\Vert^2\leq 
\lambda \int_{\partial \hat M_{\rm thick}\times [0,t_0]}f^2$.
As $\lambda <1$ by assumption, this violates the 
first inequality in the formula (\ref{integral3}). Note that this
inequality applies since by our setup, the estimate in
inequality (\ref{integral4}) holds true for $u=t_0$ and $s=0$. 
\end{proof}

\begin{corollary}\label{volume2}
Let $M$ be a hyperbolic 3-manifold and let $\hat M_{\rm thick}$ be the
tubular neighborhood of radius $96$ about $M_{\rm thick}$. 
Let $N$ be the tubular neighborhood of radius one about $\partial \hat M_{\rm thick}$
in $\hat M_{\rm thick}$. 
Let $f:\hat M_{\rm thick}\to \mathbb{R}$ be a smooth function with 
${\cal R}(f)\leq 1/96$; then
\[\int_Nf^2\leq \frac{1}{32}\int_{\hat M_{\rm thick}}f^2.\]
\end{corollary}
\begin{proof}
Parameterize $\hat M_{\rm thick}-M_{\rm thick}=W$ as $W=V\times [0,96]$
where $\partial \hat M_{\rm thick}=V\times \{0\}$. With this notation, we have
$N=V\times [0,1]$. 

We distinguish two cases. In the first case, there exists some $k\in [0,95]$
so that $\int_{V\times [0,1]}f^2\geq 3\int_{V\times [k,k+1]}f^2$. By the first
part of Lemma \ref{volumecontrol}, we have
${\cal R}(f\vert_{V\times [0,k+1]})\geq 1/3$. This implies that
\[\int_{\hat M_{\rm thick}}\Vert \nabla f\Vert^2
\geq \int_{V\times [0,k+1]} \Vert \nabla f\Vert^2\geq 
\frac{1}{3}
\int_{V\times [0,k+1]}f^2.\]

But ${\cal R}(f)\leq 1/96$, that is,
\[\int_{\hat M_{\rm thick}}\Vert \nabla f\Vert^2\leq \frac{1}{96}\int_{\hat M_{\rm thick}}f^2.\]
These two estimates together yield
\[\int_{V\times [0,1]}f^2\leq \int_{V\times [0,k+1]}f^2\leq \frac{1}{32}\int_{\hat M_{\rm thick}}f^2\]
as claimed.

In the second case, we have $\int_{V\times [0,1]}f^2\leq 3\int_{V\times [k,k+1]}f^2$ for all 
$k\leq 95$. But this implies as before that 
\[\int_{V\times [0,1]}f^2\leq \frac{1}{32}
\int_{V\times [0,96]}f^2\leq \frac{1}{32}\int_{\hat M_{\rm thick}}f^2\]
which is what we wanted to show.
\end{proof}

As a consequence, we obtain Theorem \ref{thm2} from the introduction.

\begin{proposition}\label{eigenvalue}
There is a constant $q_4>0$ with the following property.
Let $M$ be a finite volume oriented hyperbolic $3$-manifold; then for a suitable
choice of a Margulis constant, we have 
\[\lambda_k(M)\leq 
(3+q_4(\log {\rm vol}(M_{\rm thin}))\lambda_k(M_{\rm thick})\] 
for all $k$ so that $\lambda_k(M_{\rm thick})<1/96$. 
\end{proposition}
\begin{proof} 
Let $M$ be a finite volume oriented hyperbolic $3$-manifold. 
Denote by
$\hat M_{\rm thick}$ the tubular neighborhood of radius $96$ 
of the thick part
$M_{\rm thick}$ of $M$ for some choice of Margulis constant. 

Let $k\geq 1$ be such that $\lambda_k(\hat M_{\rm thick})=\lambda_k<1/96$
and let 
$f:\hat M_{\rm thick}\to \mathbb{R}$ be an eigenfunction
for the eigenvalue $\lambda_k$, with 
Neumann boundary conditions. 
Then $f$ is a smooth function on $\hat M_{\rm thick}$ 
which solves the Laplace equation 
\[\Delta(f)+
\lambda_k f=0.\]

Our goal is to extend the function $f$ on $\hat M_{\rm thick}$ 
to a function $F$ on $M$ which is contained
in the Sobolev space ${\cal H}(M)$ of square integrable
functions on $M$, with square integrable weak derivative, in 
such a way that the Rayleigh quotient of $F$ is controlled
by the Rayleigh quotient of $f$ and hence by $\lambda_k$. 



The tubular neighborhood $N$ of 
$\partial \hat M_{\rm thick}$ in $\hat M_{\rm thick}$ of radius $1$
is diffeomorphic 
to $\partial \hat M_{\rm thick}\times [0,1]$,
where the real parameter is the distance $\rho$
from the boundary 
$\partial \hat M_{\rm thick}$ of $\hat M_{\rm thick}$.  
The metric on $N$ is a warped product metric.
By Corollary \ref{volume2}, we have
\begin{equation}\label{corollary36}
\int_Nf^2\leq \frac{1}{32}\int_{\hat M_{\rm thick}}f^2.
\end{equation}


Let $m>0$ be such that 
$\int_N\Vert \nabla f\Vert^2-
m\lambda_k \int_Nf^2=0$. Then we can find a number
$\delta\in [0,1]$ so that
\begin{equation}\label{delta}
\int_{\partial \hat M_{\rm thick}\times \{\delta\}}
\Vert \nabla f\Vert^2
\leq m\lambda_k\int_{\partial \hat M_{\rm thick}\times \{\delta\}}f^2.
\end{equation}

Let $V=\hat M_{\rm thick}-\partial \hat M_{\rm thick}\times [0,\delta]$.
Then $V$ is a smooth submanifold of $\hat M_{\rm thick}$ with 
boundary $\partial V=
\partial \hat M_{\rm thick}\times \{\delta\}$.
Moreover, $M-V$ is a union of Margulis tubes and cusps as before.
Extend the restriction of $f$ to $\partial V$ as in 
Proposition \ref{tube}. This yields a function 
$F:M\to (0,\infty)$ which is continuous, smooth away from 
$\partial V$, with smooth restriction to $\partial V$.
Furthermore, $F$ is square integrable, and its 
derivative (which exists in the strong sense away from the
compact hypersurface $\partial V$) is pointwise bounded. 
Thus $F$ is contained in the Sobolev space ${\cal H}(M)$ of 
square integrable functions with square integrable 
weak derivative.

The Rayleigh quotient of $F$ can be estimated as follows. 
Let \[{\cal T}=(M-\hat M_{\rm thick})\cup N\supset M-V.\] 
Using the second part of Proposition \ref{tube}, the estimate
(\ref{delta}) and decomposing
integrals, we have 
\begin{align}\label{formula1}
\int_{\cal T}\Vert \nabla F\Vert^2&\leq \int_{{\cal T}\cap V}\Vert \nabla f\Vert^2 +
 m\lambda_k q_2(\log {\rm vol}({\cal T}))\int_{{\cal T}-V}F^2\notag \\ 
   &\leq  m\lambda_k (1+\frac{1}{2}q_2
       \log {\rm vol}({\cal T}))
             \int_N f^2\end{align}
where the last inequality uses the definition of the constant $m>0$. 

Recall that the Rayleigh quotient of the function $f$ on 
$\hat M_{\rm thick}$ equals $\lambda_k$.
We compute
\begin{align}\label{step5}
  \int_M\Vert \nabla F\Vert^2=&
        \int_{M-{\cal T}}\Vert \nabla F\Vert^2+\int_{{\cal T}}
\Vert \nabla F\Vert^2
 \leq \int_{\hat M_{\rm thick}}\Vert \nabla f\Vert^2+
\int_{\cal T}\Vert \nabla F\Vert^2\\
&\leq 
\lambda_k\bigl(\int_{\hat M_{\rm thick}}f^2 +
m(1+\frac{1}{2}q_2\log  {\rm vol}({\cal T}))
                                    \int_Nf^2\bigr)
                                    \notag
\end{align}
where the last inequality follows from (\ref{formula1}). 

As
$\int_N\Vert \nabla f\Vert^2=m\lambda_k \int_N f^2$, we have
$\int_Nf^2\leq \frac{1}{m}\int_{\hat M_{\rm thick}} f^2$. 
Inserting into the estimate (\ref{step5}) 
implies that 
\[\int_{M}\Vert \nabla F\Vert^2\leq 
\lambda_k(2+\frac{1}{2}q_2\log {\rm vol}({\cal T}))\int_{\hat M_{\rm thick}}f^2.\]

To complete the control on the Rayleigh quotient of $F$, we compare
the square norm of $F$ to the square norm of $f$. For convenience, 
extend the eigenfunction $f$ on $\hat M_{\rm thick}$ to all of $M$ by $0$.
As $F=f$ on $V=\hat M_{\rm thick}-\partial \hat M_{\rm thick}\times [0,\delta]$,
using inequality (\ref{corollary36}) we estimate
\begin{equation}\label{normcompare}
  \frac{31}{32}\int_{\hat M_{\rm thick}}f^2 \leq
  \int_{V}f^2
    \leq \int_Vf^2+\int_{M-V}F^2=\int_MF^2.
  \end{equation}
We deduce that 
the Rayleigh 
quotient of $F$ is bounded from above by  
\[(3+q_2(\log {\rm vol}(M_{\rm thin}))\lambda_k(\hat M_{\rm thick}).\]
This is the property of the extension function $F$ we were aiming at.
Note also for later reference that as another consequence of 
Lemma \ref{volumecontrol} and the fact that $F=f$ on $V$, we obtain
\begin{equation}\label{squarenorm2}
\int_M(f-F)^2=\int_{M-V}F^2\leq \frac{1}{2}\int_{\partial 
\hat M_{{\rm thick}\times \{\delta\}}}f^2.
\end{equation}

To summarize, for an eigenfunction $f_k$ on $\hat M_{\rm thick}$
with Neumann boundary conditions and eigenvalue
$\lambda_k<1/96$, we constructed an extension $F_k:M\to \mathbb{R}$
with controlled Rayleigh quotient as predicted in the proposition.

Now let us assume that for each $i\leq k$ we constructed from an 
eigenfunction $f_i$ on $\hat M_{\rm thick}$ with eigenvalue $\lambda_i(\hat M_{\rm thick})$
with the above procedure the function $F_i$. 
Let $E_{k-1}$ be the linear span of the functions
$F_i$ for $i<k$. Assume that the function $F_k$ is normalized 
so that $\int F_k^2=1$. Using the above notations with $\delta_i\in [0,1]$
the number which enters the construction of $F_i$, 
assume furthermore for the moment
that 
\begin{equation}\label{integralestimate}
\int_{\partial \hat M_{\rm thick}\times \{\delta_k\}}f_k^2\leq 
\int_Nf_k^2\leq \frac{1}{32}\int_{\hat M_{\rm thick}}f_k^2.
\end{equation}
Using inequalities (\ref{normcompare}) and (\ref{squarenorm2}), we then
have 
\[\int_{\cal T}F_k^2\leq \frac{3}{32}\int_M F_k^2.\]

The function $F_k$ may not be orthogonal to $E_{k-1}$ for the 
$L^2$-inner product. Denote by $H$ the $L^2$-orthogonal 
projection of $F_k$ to $E_{k-1}$.
Note that $\int_M H^2<1$. 
%
Since $H\in E_{k-1}$ there exists a finite linear combination 
$h$
of eigenfunctions on $\hat M_{\rm thick}$ with
Neumann boundary conditions and eigenvalues $\lambda_i(\hat M_{\rm thick})$ 
for $i<k$ 
which gives rise to $H$
with the above construction. As $\int_M H^2<1$, by construction 
 $\int_{\hat M_{\rm thick}} h^2\leq \frac{32}{31}$, and 
Corollary \ref{volume2} shows that
\[\int_Nh^2\leq \frac{1}{32}\int_{\hat M_{\rm thick}}h^2
\leq \frac{1}{31}.\]

Now $F_k=f_k$ and $H=h$ on $\hat M_{\rm thick}-N$ and furthermore
$\int_{\hat M_{\rm thick}}f_kh=0$ and hence
\begin{align}\label{computation10}
\int_M (F_k-H)^2 &\geq \int_{\hat M_{\rm thick}-N}(F_k-H)^2
=\int_{\hat M_{\rm thick}-N}f_k^2-2f_kh+h^2\\
& \geq \frac{29}{32} -2 \vert \int_N f_kh \vert \geq 
\frac{29}{32}-\frac{2}{31} \geq \frac{26}{32}\notag
\end{align}
by the Cauchy Schwarz inequality. 

By the above construction, there exists a 
universal constant $c>0$ so that 
\[\int_M\Vert \nabla F_k\Vert^2\leq \lambda_k(3+
c\log {\rm vol}(M_{\rm thin}))\int_{\hat M_{\rm thick}} F_k^2\] 
and the same holds true for $H$, with $\lambda_k$ replaced by
$\lambda_{k-1}$. But this implies that 
the Rayleigh quotient of $F_k-H$ is bounded
from above by a fixed multiple of $\lambda_k(M_{\rm thick})(1+\log{\rm vol}(M_{\rm thin}))$.
As $k$ with $\lambda_k<1/96$ was arbitrary, 
the proposition now follows from Rayleigh`s principle
(see \cite{C84} for details) provided that we can assure that
inequality (\ref{integralestimate}) holds true for all $k$.

%
%


The final step of this proof consists in modifying the construction of the function $F$
from an eigenfunction $f$ on $\hat M_{\rm thick}$ with Neumann boundary conditions
so that $F$ fulfills inequality (\ref{integralestimate}).

To this end recall from the second part of Lemma \ref{volumecontrol} 
that 
\[\int _{\partial \hat M_{\rm thick}\times \{0\}}f^2\leq \int_Nf^2.\]

Let $\beta:[0,\delta]\to 
\mathbb{R}$ be such that the function $u:\partial \hat M_{\rm thick}\times [0,\delta]=W
\to (0,\infty)$ defined by $u(x,s)=e^{\beta(s)}f(x,\delta)$ satisfies 
\[\int_{\partial \hat M_{\rm thick}\times \{s\}}u^2=
\int_{\partial \hat M_{\rm thick}\times \{s\}}f^2 \text{ for all }s\in [0,\delta].\]
Clearly we have $\int_Wu^2=\int_Wf^2$. 

Let $\nu$ be the vector field on $W$ which equals the 
exterior normal field of the hypersurfaces $\hat M_{\rm thick}\times \{s\}$, defining the 
orientation of the interval. We claim that 
\begin{equation}\label{w}
\int_W \nu(u)^2\leq \int_W\nu(f)^2.\end{equation}

To show the claim we evoke formula (\ref{areaderi}) from the proof of
Lemma \ref{volumecontrol} which gives

\[\frac{d}{dt}\int_{\partial \hat M_{\rm thick}\times \{t\}}f^2=
c(t)\int_{\partial \hat M_{\rm thick}\times \{t\}}f^2+2\int_{\partial \hat M_{\rm thick}\times \{t\}}f\nu(f).\]
By the definition of the function $u$, for all $t\in [0,\delta]$ we therefore have
\[\int_{\partial \hat M_{\rm thick}\times \{t\}}f\nu(f)=\int_{\partial \hat M_{\rm thick}\times \{t\}}u\nu(u).\]

But $\nu(u)(x,t)=\beta^\prime(t)u(x,t)$ and hence
\[\int_{\partial \hat M_{\rm thick}\times \{t\}}u\nu(u)=\beta^\prime(t) 
\int_{\partial \hat M_{\rm thick}\times \{t\}}u^2,\]
furthermore 
\begin{align}\label{estimate10}
\int_{\partial \hat M_{\rm thick}\times \{t\}}\nu(u)^2&=
\vert \beta^\prime(t)\vert^2\int_{\partial \hat M_{\rm thick}\times \{t\}}u^2\notag\\
&=(\int_{\partial \hat M_{\rm thick}\times \{t\}}u\nu(u))^2/
\int_{\partial \hat M_{\rm thick}\times \{t\}}u^2.\end{align}

Observe that the last line in equation (\ref{estimate10})
coincides with the corresponding 
expression for the function $f$. By the Schwarz' inequality, 
\[\vert \int_{\partial \hat M_{\rm thick}\times \{t\}}
  f\nu(f)\vert \leq
  (\int_{\partial \hat M_{\rm thick}\times \{t\}} f^2)^{1/2}
  (\int_{\partial \hat M_{\rm thick}\times \{t\}} \nu(f)^2)^{1/2}.\]
On the other hand, by (\ref{estimate10}), for the
function $u$ equality holds in this inequality. This yields 
$\int_W \nu(u)^2\leq \int_W \nu(f)^2$ as claimed. 

Following the discussion in the proof of Proposition \ref{tube}, for 
$s,t\in [0,\delta]$ the radial 
projection of the hypersurface
$\partial \hat M_{\rm thick}\times \{s\}$ onto the 
hypersurface 
$\partial \hat M_{\rm thick}\times \{t\}$ is bilipschitz, with uniformly bounded 
bilipschitz constant not depending on $M$. 
This implies that there is a universal constant
$c>0$ such that for each $s\in [0,\delta]$, 
the Rayleigh quotient
of the restriction of $u$ to
$\partial \hat M_{\rm thick}\times \{s\}$ is not bigger than
$c$ times the Rayleigh quotient of the restriction of $u$ to 
$\partial \hat M_{\rm thick}\times \{\delta\}$.
Together with the estimate (\ref{w})
on normal derivatives and the choice of
$\delta$, we obtain 
\[\int_{\partial \hat M_{\rm thick}\times [0,\delta]}
  \Vert \nabla u\Vert^2
  \leq cm\lambda_k\int_{\partial \hat M_{\rm thick}\times
    [0,\delta]}f^2
  +\int_{\partial \hat M_{\rm thick}\times [0,\delta]}
  \Vert \nabla f\Vert^2.\]

Use Proposition \ref{tube} to 
extend the restriction of the
function $u$ to $\hat M_{\rm thick}\times \{0\}$ to 
the complement of 
$\hat M_{\rm thick}$ in $M$
(which consists of a collection of Margulis tubes and cusps).
It now follows from the beginning of this proof that the
resulting function $F$ has all properties
predicted in
the proposition.

To provide more details of this estimate, let 
again ${\cal T}=M-\hat M_{\rm thick}\cup N$. 
Then $f=F$ on $M-{\cal T}$, and by Proposition \ref{tube},
Lemma \ref{volumecontrol}
and inequality (\ref{corollary36}), 
\begin{equation}\label{estimate12}
  \int_{\cal T}F^2\leq \frac{3}{2}\int_Nf^2\leq
\frac{3}{64}\int_{\hat M_{\rm thick}}f^2. \end{equation}
This implies as in the estimate (\ref{normcompare}) that 
\begin{equation}\label{comparisoncontrol2}
\int_Mf^2\leq \int_MF^2\leq \frac{67}{64}\int_Mf^2.
\end{equation}
On the other hand, using the estimate (\ref{delta})
and the construction of $F$, 
\[\int_{{\cal T}}\Vert \nabla F\Vert^2\leq
  \lambda_kq_2m(\log{\rm vol}({\cal T})+c)\int_{N}f^2\]
which yields the required estimate on Rayleigh quotients as above.
\end{proof}

\begin{remark} For a compact 
hyperbolic manifold $M$ of dimension $n\geq 4$, 
Burger and Schroeder 
\cite{BS87} proved the upper bound
\[\lambda_1(M)\leq \frac{\alpha_n+\beta_n\log{\rm vol}(M)}{{\rm diam}(M)}\]
with constants $\alpha_n,\beta_n$ depending only on $n$.
As hyperbolic $3$-manifolds with cusps can be Dehn-filled, with 
fixed volume bound,  this result is in general false for 
hyperbolic $3$-manifolds.
\end{remark}


\bigskip

\noindent
Universit\"at Bonn, Endenicher Allee 60, 53115 Bonn, Germany\\
e-mail: ursula@math.uni-bonn.de

\end{document}